\documentclass[conference,letterpaper]{IEEEtran}

\ifCLASSINFOpdf
\else
\fi
\usepackage{float}
\usepackage{graphicx}
\usepackage{subcaption}
\usepackage{multicol}
\usepackage{algorithm}
\usepackage{algpseudocode}
\usepackage{amsmath}
\usepackage{amssymb}
\usepackage{mathrsfs}
\usepackage{amsthm}
\usepackage{bm}
\usepackage{cite}
\usepackage{enumerate}
\usepackage{diagbox}
\usepackage{color}
\usepackage{comment}

\newtheorem{theorem}{Theorem}
\newtheorem{definition}{Definition} % defintion and theorem has the consistent no.

\newtheorem{corollary}{Corollary}
\newtheorem{lemma}{Lemma}

\begin{document}
\title{Optimal Trajectories of a UAV Base Station Using Lagrangian Mechanics}

\author{\IEEEauthorblockN{Marceau Coupechoux\IEEEauthorrefmark{1}, J\'er\^ome Darbon \IEEEauthorrefmark{2}, Jean-Marc K\'elif\IEEEauthorrefmark{3}, and Marc Sigelle\IEEEauthorrefmark{4}}
\IEEEauthorblockA{\IEEEauthorrefmark{1}Telecom ParisTech, France, \IEEEauthorrefmark{2}Brown University, US, 
\IEEEauthorrefmark{3}Orange Labs, France, \IEEEauthorrefmark{4}On leave from Telecom ParisTech, France}
\IEEEauthorblockA{Email:  marceau.coupechoux@telecom-paristech.fr, jerome\_darbon@brown.edu, jeanmarc.kelif@orange.com,
marc.sigelle@gmail.com}
}
\maketitle

\begin{abstract}
In this paper, we consider the problem of optimizing the trajectory of an Unmanned Aerial Vehicle (UAV) Base Station (BS). We consider a map characterized by a traffic intensity of users to be served. The UAV BS must travel from a given initial location at an initial time to a final position within a given duration and serves the traffic on its way. The problem consists in finding the optimal trajectory that minimizes a certain cost depending on the velocity and on the amount of served traffic. We formulate the problem using the framework of Lagrangian mechanics. When the traffic intensity is quadratic (single-phase), we derive closed-form formulas for the optimal trajectory. 
%When the traffic map is divided in regions of quadratic traffic intensity (multi-phase) and time-varying, we propose an online algorithm based on Model Predictive Control (MPC). 
When the traffic intensity is bi-phase, we provide necessary conditions of optimality and propose an Alternating Optimization Algorithm that returns a trajectory satisfying these conditions. The Algorithm is initialized with a Model Predictive Control (MPC) online algorithm. Numerical results show how we improve the trajectory with respect to the MPC solution.    
\end{abstract}

\section{Introduction}
Unmanned Aerial Vehicles (UAV) are expected to play an increasing role in future wireless networks\footnote{J. Darbon is supported by NSF DMS-1820821.}~\cite{Zeng16b}. UAVs may be deployed in an ad hoc manner when the traditional cellular infrastructure is missing. They can serve as relays to reach distant users outside the coverage of wireless networks. They also may be used to disseminate data to ground stations or collect information from sensors. 
In this paper, we address one of the envisioned use cases for UAV-aided wireless communications, which relates to cellular network offloading in highly crowded areas~\cite{Zeng16b}. More specifically, we focus on the path planning problem or trajectory optimization problem that consists in finding an optimal path for a UAV Base Station (BS) that minimizes a certain cost depending on the velocity and on the amount of served traffic. Our approach is based on the Lagrangian mechanics framework. 

\subsection{Related Work}

%\cite{Yang18} studies a system in which a UAV shall collect data from a ground station. If the UAV flies closer to the station, the transmission power required to send data is reduced, while the propulsion energy to approach the station increases. The paper studies this trade-off and consider  either circular or straight trajectories. Remark: no real trajectory optimization.
%\cite{fotouhi2016dynamic} Model: there is a dynamic traffic to be served; time is slotted; the drone has a fixed speed but can change direction to reposition itself according to the traffic; the goal to maximize the spectral efficiency while taking into account the energy consumption of the drone. Contribution: Heuristic online algorithms to dynamically reposition the drone BS in order to reduce the distance between the BS and the UE and thus to improve the spectral efficiency. Remark: good point is that algorithms are online but there is no optimality guarantee.

UAV trajectory optimization for networks has been tackled maybe for the first time in~\cite{Pearre10}. The model consists in 
a UAV flying over a sensor network from which it has to collect some data. 
%The sequence of sensors to be visited over a tour is known. The radio field is unknown and imperfections of the auto-pilot are taken into account. 
The problem consists in optimizing the trajectory length of the UAV under the constraint that it collects the required amount of data from every sensor. 
%It is related to the route design problem for {\it data mules} or {\it message ferries} in ad hoc networks~\cite{Tariq06}. 
Authors use a reinforcement learning approach where improved trajectories are sequentially learned over several tour iterations. This model is different from ours as it allows the UAV to learn the optimal trajectory from previous experience. 
The problem of optimally deploying UAV BSs to serve traffic demand has been addressed in the literature by considering static UAVs BSs or relays, see e.g.~\cite{bor2016efficient, sharma2016uav}. The goal is to optimally position the UAV so as to maximize the data rate with ground stations or the number of served users. In a very recent work~\cite{Yang18}, a data rate-energy trade-off is studied. In these works the notion of trajectory is either ignored or restricted to be circular or linear. 
\begin{figure}[t]
\begin{center}
\includegraphics[width=0.8\linewidth]{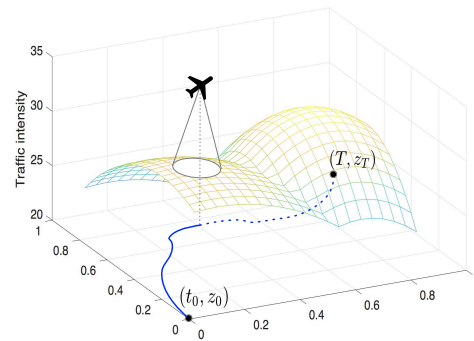}
\end{center}
\caption{\label{fig:systemmodel}A UAV Base Station travels from $z_0$ at $t_0$ to $z_T$ at $T$ and serves a user traffic characterized by its intensity.}
\end{figure}
In robotics and autonomous systems, trajectory optimization is known as {\it path planning}~\cite{Mac16}. In this aim, there are classical methods like Cell Decomposition, Potential Field Method or Probabilistic Road Map and there are heuristic approaches, e.g. bio-inspired algorithms. Authors of \cite{chi2012civil} have capitalized on this literature and proposed a path planning algorithm for drone BSs based on A* algorithm. The main goal of these papers is to reach a destination while avoiding obstacles, and in \cite{chi2012civil} the speed cannot be controlled. In our work, we intend to minimize a certain cost function along the trajectory by controlling the velocity of the UAV.
This goal is studied in optimal control theory~\cite{Liberzon11} and is applied for example in aircraft trajectory planning~\cite{Delahaye14}. Most numerical methods in control theory can be classified in {\it direct} and {\it indirect} methods. Indirect methods provide analytical solutions from the calculus of variations and use first order necessary conditions for a trajectory to be optimal. In direct methods, the problem is transformed in a non linear programming problem using discretized time, locations and controls. 
Direct methods are heavily applied in a series of very recent publications in the field of UAV-aided communications. In \cite{Zeng16} for example, a UAV relay assists the communication between a source and a destination. As the resulting problem is non-convex, it is first approximated and then solved by successive convex optimization. In \cite{Zeng17}, the objective is to maximize the energy efficiency of a UAV-to-ground station communication by taking into account the propulsion energy consumption and by optimizing the trajectory. Again, sequential convex optimization is applied to an approximated problem. In the same vein, \cite{wu2018joint} considers multiple-UAV BSs used to serve fixed users. The quality of the solution to the nonlinear program may heavily depend on the initial guess. Authors thus propose an heuristic based on circular trajectories to initialize their algorithm. 
 With direct methods, because of the discretization, the differential equations and the constraints of the systems are satisfied only at discrete points. This can lead to less accurate solutions than indirect methods and the quality of the solution  depends on the quantization step~\cite{vonStryk92}. Although every iteration of the sequential convex optimization technique has a polynomial time complexity, practical resolution time may dramatically increase with the quantization grid and the dimension of the problem. 
We thus propose in this paper an indirect approach based on Lagrangian mechanics that has the advantage to provide closed-form expressions for the optimal trajectories when the potential is quadratic (we say \emph{single-phase}). When the potential is quadratic by region (or \emph{multi-phase}) the optimization process consists in finding the right crossing time and location on the interface of the regions. This question is an active field of research in control theory, see e.g.~\cite{Barles16}. As explained in \cite{Betts98,Patterson14}, a trajectory optimization problem can be decomposed in different {\it phases} or {\it arcs}. Phases are sequential in time, i.e., they partition the time domain. Differential equations describing the system dynamics cannot change during a phase. This point of view allows us to consider the multi-phase problem.

\subsection{Contributions}
Our contributions are the following:
\begin{itemize}
	\item {\it Problem Formulation:} To the best of our knowledge, this is the first time that the UAV BS trajectory problem is formulated using the formalism of Lagrangian mechanics. This approach provides closed-form equations when the potential is quadratic and thus very low complexity solutions compared to existing solutions in the literature.
    \item {\it Closed-form expression of the optimal trajectory with single phase traffic intensity:} When the traffic intensity map is made of a single hot spot or traffic hole, has a quadratic form ({\it single phase}), and is time-independent, closed form expressions for the optimal trajectory are derived. It consists in a part of hyperbole for a hot spot and corresponds to a repulsor in mechanics. For a traffic hole, the trajectory is on an ellipse and corresponds to the case of an attractor in mechanics.  
    \item {\it Characterization of the optimal solution in multi-phase traffic intensity:} When the traffic map has several hot spots or traffic holes ({\it multi-phase}) whose regions are separated by interfaces and is time-independent, we derive necessary conditions to be fulfilled by the position and the instant at which the optimal trajectory crosses an interface (see Theorem~\ref{th:multiphase}).  
    \item {\it An online algorithm for multi-phase time-varying traffic intensity:} When the traffic map is multi-phase and is time-varying, we propose an online algorithm based on MPC. %Although there is no guarantee of optimality for MPC, this algorithm has the advantage of being online and to have $O(1)$ complexity at every iteration (see Algorithm~\ref{alg:mpc}). 
    \item {\it An Alternating Optimization Algorithm for bi-phase time-independent traffic intensity:} When the traffic intensity is made of two hot spots separated by an interface ({\it bi-phase}) and is time-independent, we propose an Alternating Optimization Algorithm that finds a stationary point for the cost function. This algorithm has a complexity $O(1)$ at every iteration, whereas iterations of the sequential convex optimization technique have polynomial time complexity (see Algorithm~\ref{alg:aoa}). 
\end{itemize}

The paper is structured as follows. In Section~\ref{sec:systemmodel} we give the system model and its interpretation in terms of Lagrangian mechanics. In Section~\ref{sec:lagrangianformulation}, we formulate the problem and give preliminary results. Section~\ref{sec:quadcost} is devoted to the characterization of the optimal trajectories. Section~\ref{sec:algos} presents our algorithms and Section~\ref{sec:conclusion} concludes the paper.  

{\bf Notations:} %We denote $x\cdot y$ the inner product of $x$ and $y\in\mathbb{R}^2$; $||x||$ is the Euclidean norm of vector $x\in\mathbb{R}^2$. 
Let $f:\mathbb{R}^n\times \mathbb{R}^m \to  \mathbb{R}$ defined by $f(x,y)$ where $x = (x_1, \dots,x_n)\in \mathbb{R}^n$ and $y =(y_1,\dots,y_m)\in \mathbb{R}^m$. Let $a \in \mathbb{R}^n$ and $b \in \mathbb{R}^m$. We denote by $\frac{\partial f}{\partial x_i} (a,b)$ the partial derivative of $f$ with respect to the variable $x_i$ at $(a,b) \in \mathbb{R}^n \times \mathbb{R}^m$. We also introduce the notations $\nabla _x f(a,b) = ( \frac{\partial f}{\partial x_1}(a,b), \dots, \frac{\partial f}{\partial x_n}(a,b)) \in \mathbb{R}^n$ and $\nabla_y f(a,b) = ( \frac{\partial f}{\partial y_1}(a,b), \dots, \frac{\partial f}{\partial y_m}(a,b)) \in \mathbb{R}^m$.

\section{System Model and Interpretation} \label{sec:systemmodel}
\subsection{System Model}
We consider a network area characterized by a traffic density  at position $z$ and time $t$. We intend to control the trajectory and the velocity of a UAV base station, which is located in $z_0\triangleq z(t_0)$ at $t_0$ and shall reach a destination $z_T\triangleq z(T)$ at $T$ with the aim of minimizing a cost determined by the velocity and the traffic, defined hereafter by (\ref{eq:cost}). At $(t,z)$, we assume that the UAV BS is able to cover an area, from which it can serve users (see Figure~\ref{fig:systemmodel}). The velocity of the UAV BS induces an energy cost. In this model, we control the velocity vector $a$ of the UAV BS.
The general form of the cost function is as follows
\begin{equation} \label{eq:cost}
\mathcal{L}(t,z,a)=\frac{K}{2}||a||^2-u(t,z)
\end{equation}
where the first term is a cost related to the velocity of the vehicle ($K$ is a positive constant), and $\|\cdot\|$ denotes the usual Euclidean norm. The higher is the speed, the higher is the energy cost. The second term is a {\it user traffic intensity}, i.e., the amount of traffic served by the UAV BS at $(t,z)$. Note that a non-zero energy at null speed can be incorporated in the model by adding a constant. Without loss of generality, we assume that this constant is null.

\section{Lagrangian Mechanics Formulation} \label{sec:lagrangianformulation}

\subsection{Problem Formulation}

Let $S(t_0, z_0,T,z_T)$ be the minimal total cost along any trajectory between $z_0$ at $t_0$ and $z_T$ at $T$ (also called {\it the action} in mechanics or {\it value function} in control theory). 
%and $\mathcal{L}_T(z_T)$ be a terminal cost at $T$. 
Let us define $\Omega(t_0,T)$ as the space of absolutely continuous functions from $[t_0;T]$ to $\mathbb{R}^2$. Our problem can now be formulated as follows
\begin{eqnarray} \label{eq:problem}
S(t_0, z_0,T,z_T) \!\!\!&\!=\!\!\!\!&\!\!\!\!\! \min_{a\in \Omega(t_0,T) }\!\!\int_{t_0}^T\!\!\!\!\mathcal{L}(s,z(s),a(s))ds \!\!+\!\! J(z(T))
\end{eqnarray}
where $\frac{d z}{d t}(t) = a(t)$, $z(t_0) = z_0$, and 
%\begin{equation} \label{eq:diffeq}
%\begin{cases}
%\frac{d z}{d t}(t) = a(t) \\
%z(t_0) = z_0
%\end{cases}
%\end{equation}
$J$ is the terminal cost defined by $J(z)=0$ if $z=z_T$ and $J(z)=+\infty$ otherwise.
%\begin{equation}
%J(z) = \begin{cases}
%0 & \text{if } z=z_T \\
%+\infty & \text{otherwise}
%\end{cases}
%\end{equation}
For simplicity reasons, we assume the existence and uniqueness of the optimal control $a^*(t)$ in \eqref{eq:problem} and denote the associated optimal trajectory $z^*(t)$. In a traffic map symmetric with respect to $z_0$ and $z_T$, the reader can convince himself that the uniqueness is not guaranteed. 
%we define the value function as a function of $t$ and $z$ (hence $z=z^*(t)$) as follows:
% \begin{equation}
% S(t,z;t_0, z_0,T,z_T)=\int_{t_0}^t \mathcal{L}(s,z^*(s),a^*(s))ds.
% %+ J(z_T).
% \end{equation}
% In this notation, the variables are $z$ and $t$, whereas $z_0$, $t_0$, $z_T$ and $T$ are the parameters of the system. We assume that these parameters are fixed, i.e., not free for the optimization. When it is clear from the context, we will omit the parameters in the notation. With a slight abuse of notation, we can remark that $S(T,z_T;t_0, z_0,T,z_T) = S(t_0, z_0,T,z_T)$. The function $\mathcal{L}$ is called the \emph{Lagrangian} in mechanics and the \emph{running cost} in control theory. 

\subsection{Preliminary Results From Lagrangian Mechanics}
We provide in this section important results from the Lagrangian mechanics for the convenience of the reader. %We recall the proofs in appendix for the sake of completeness.   

\begin{definition}[Impulsion]
The \emph{impulsion} function is defined as 
\begin{equation}
p(t,z,a):=\nabla_a \mathcal{L}(t,z,a)
%\frac{\partial \mathcal{L}}{\partial a}(t,z,a).
\end{equation}
\end{definition}
In the Newtonian classical framework that is used here (see \eqref{eq:cost}), the impulsion is the product of the particle mass by its velocity (hence the standard term "impulsion").
\begin{definition}
The \emph{Hamiltonian} function is defined as
\begin{equation}
H(t,z,p):=\max_{a\in\mathbb{R}^2} p\cdot a-\mathcal{L}(t,z,a).
\end{equation}
% We also define the function $H$ as follows:
% \begin{equation}
% H(t,z,p):=\inf_a \mathcal{H}(t,z,a,p).
% \end{equation}
\end{definition}
%We emphasize that the Lagrangian $\mathcal{L}$, the impulsion $p$ and the Hamiltonian $H$ are three functions that depend from three independent variables. 
%, $z$ and $a$ are treated as independent variables when computing the partial derivatives of $\mathcal{L}$. The Hamiltonian is defined here as a general function of three variables. %When there exists an optimal trajectory, the infimum in the definition of $H$ is a minimum.  

\begin{lemma}[Euler-Lagrange Equations] \label{lemma:EulerLagrange}
Along the optimal trajectory $z^*(t)$ that starts from $z_0$ at $t_0$ and ends at $z_T$ at $T$, we have
% \begin{equation} \label{eq:eulerlagrange}
% \frac{d}{dt}\frac{\partial \mathcal{L}}{\partial \dot{z}}=\frac{\partial \mathcal{L}}{\partial z}
% \end{equation}
% or equivalently:
\begin{equation}\label{eq:eulerlagrange2}
\frac{d}{dt}\nabla_a \mathcal{L}(t,z^*(t),a^*(t))=\nabla_z \mathcal{L}(t,z^*(t),a^*(t))
\end{equation}
or equivalently
\begin{equation} \label{eq:eulerimpulsion}
\frac{dp}{dt}(t,z^*(t),a^*(t))=\nabla_z\mathcal{L}(t,z^*(t),a^*(t))
\end{equation}
% Moreover:
% \begin{equation} \label{eq:eulerlagrange3}
% \nabla_a\mathcal{L}(T,z_T,a^*(T))=-\nabla J(z_T)
% \end{equation}
\end{lemma}

\begin{proof}
See Appendix~\ref{app:EulerLagrange}.
\end{proof}

The Euler-Lagrange equation is the first-order necessary condition for optimality and holds for every point on the optimal trajectory.  %From now on, we assume that the final position $z_T$ is fixed and the Lagrangian is time-independent, i.e., $\mathcal{L}(z,t,a)=\mathcal{L}(z,a)$.

\begin{lemma} \label{lemma:lagrangienhomogene}
If the Lagrangian $\mathcal{L}(t,z,a)$ is time-independent and $\alpha$-homogeneous in $z$ and $a$ for $\alpha>0$, i.e., $\mathcal{L}(\lambda z,\lambda a)=|\lambda|^{\alpha}\mathcal{L}(z,a)$ for all $\lambda\in \mathbb{R}$, $S$ given by \eqref{eq:problem} reads
\begin{equation} \label{eq:valuefunction}
S(t_0, z_0,T,z_T) = \frac{1}{\alpha}[z\cdot p]_{t_0}^T+J(z_T).
\end{equation}
\end{lemma}

\begin{proof}
See Appendix~\ref{app:lagrangienhomogene}.
\end{proof}

\begin{lemma}[Hamilton-Jacobi] \label{lemma:hamiltonjacobi} Along the optimal trajectory, we have for $t\in(t_0;T)$
\begin{equation}
\frac{\partial S}{\partial t_0}(t,z^*(t),T,z_T)=H(t,z^*(t),-p^*(t))
\label{Hamilton-Jacobi-backward-ms:eq}
\end{equation}
% * <marc.sigelle@gmail.com> 2018-07-12T14:48:37.465Z:
%
% ^.
% * <marc.sigelle@gmail.com> 2018-07-12T14:48:36.127Z:
%
% ^.
\begin{equation}
\frac{\partial S}{\partial T}(t_0,z_0,t,z^*(t))=-H(t,z^*(t),p^*(t))
\label{Hamilton-Jacobi-forward-ms:eq}
\end{equation}
where 
\begin{equation} \label{eq:optimpulsion}
p^*(t)=\nabla_a\mathcal{L}(t,z^*(t),a^*(t))=\nabla_z S(t_,z^*(t),T,z_T)
\end{equation}
\end{lemma}

\begin{proof}
See Appendix~\ref{app:hamiltonjacobi}.
\end{proof}

From now, we assume that the Lagrangian is time-independent, i.e., $\mathcal{L}(t,z,a)=\mathcal{L}(z,a)$, and is an even function in $a$, i.e., $\mathcal{L}(z,-a)=\mathcal{L}(z,a)$. A direct consequence is that $H$ is time-independent and is an even function in $p$, i.e., we write $H(t,z,p)=H(z,p)$ and $H(z,-p)=H(z,p)$. 

\section{Optimal Trajectory} \label{sec:quadcost}
In this section, we characterize the optimal trajectory when the traffic intensity is a quadratic form and also when it is made of two regions of quadratic form separated by an interface\footnote{We leave for further work the way to approximate a realistic traffic intensity map by a set of regions with intensities of quadratic form.}. We call these two cases {\it single-phase} and {\it multiple-phase} intensities respectively. Both cases satisfy our assumptions on the Lagrangian with $\alpha=2$.  

\subsection{Single-Phase Optimal Trajectory}

Assume that the traffic intensity is of the form $u(z)=\frac{1}{2}u_0||z||^2$. When $u_0>0$, this function models a traffic hole in $z=0$. When $u_0<0$, it models a traffic hot spot at $z=0$. We disregard the case $u_0=0$ because it corresponds to a constant traffic intensity that is not of interest in this paper. Thus the cost function has  the following form
\begin{equation} \label{eq:Lsinglephase}
\mathcal{L}(z,a)=\frac{1}{2}K||a||^2-\frac{1}{2}u_0||z||^2
\end{equation}
Note that 
\begin{equation} \label{eq:impulsionquad}
p(z,a) = \nabla_a\mathcal{L}(z,a)=Ka
\end{equation}

\subsubsection{Trajectory Equation} In the single phase case, we have a closed form expression of the trajectory. 

\begin{theorem} \label{th:quadraticfunction}
If $u_0<0$, the cost function is given by (\ref{eq:vu0neg}),
\begin{figure*}
   \begin{align} \label{eq:vu0neg}
S(t_0,z_0,T,z_T)=\frac{K\omega}{2\sinh \omega (T-t_0)}\left( (|z_0|^2+|z_T|^2)\cosh \omega (T-t_0)-2z_0\cdot z_T\right)+J(z_T) %\(\)
\end{align}
    \hrulefill
\end{figure*}   
%\begin{equation}
%v(z_0,t_0,z_T, T)=\frac{\omega}{2\sinh \omega (T-t_0)}\left( (|z_0|^2+|z_T|^2)\cosh \omega (T-t_0)-2z_0\cdot z_T\right)+c_T(z_T),
%\end{equation}
the optimal trajectory is
\begin{equation} \label{eq:trajhotspot}
z^*(t)=\frac{z_T\sinh (\omega(t-t_0))+z_0\sinh (\omega(T-t))}{\sinh (\omega(T-t_0))}
\end{equation}
and the control is given by
\begin{equation} \label{eq:speedhotspot}
a^*(t)=\omega \frac{z_T\cosh(\omega(t-T))-z_0\cosh(\omega(T-t))}{\sinh(\omega(T-t_0))}
\end{equation}
where $\omega^2=-\frac{u_0}{K}$. \\
If $u_0>0$, the cost function is given by (\ref{eq:vu0pos}),
\begin{figure*}
   \begin{align} \label{eq:vu0pos}
S(t_0,z_0,T,z_T)=\frac{K\omega}{2\sin \omega (T-t_0)}\left( (|z_0|^2+|z_T|^2)\cos \omega (T-t_0)-2z_0\cdot z_T\right)+J(z_T)
\end{align}
    \hrulefill
\end{figure*}   
%\begin{equation}
%v(z_0,t_0,z_T, T)=\frac{\omega}{2\sin \omega (T-t_0)}\left( (|z_0|^2+|z_T|^2)\cos \omega (T-t_0)-2z_0\cdot z_T\right)+c_T(z_T),
%\end{equation}
the optimal trajectory is
\begin{equation} \label{eq:ellipse}
z^*(t)=\frac{z_T\sin (\omega(t-t_0))+z_0\sin (\omega(T-t))}{\sin (\omega(T-t_0))}
\end{equation}
and the control is given by
\begin{equation}
a^*(t)=\omega \frac{z_T\cos (\omega(t-t_0))-z_0\cos (\omega(T-t))}{\sin (\omega(T-t_0))}
\end{equation}
where $\omega^2=\frac{u_0}{K}$. 
\end{theorem}

\begin{proof}
See Appendix~\ref{app:quadraticfunction}.
\end{proof}

% \begin{remark}
% If $u_0<0$ and $J(z_T)$ is convex in $z_T$, then the value function is convex and can be easily optimized in $z_T$. If $u_0>0$ and $J(z_T)$ is convex in $z_T$, then we can always find $T$ sufficiently close to $t_0$ such that the value function is convex in~$z_T$.  
% \end{remark}

\begin{corollary} \label{cor:quad}
If the user traffic intensity is of the form $u(t,z)=\frac{1}{2}u_0||z||^2+u_0z\cdot e +u_1$ with $u_1\in \mathbb{R}$ and $e\in \mathbb{R}^2$, then define $\tilde{z}=z+e$, $\tilde{z}_0=z_0+e$, $\tilde{z}_T=z_T+e$ and trajectories given in Theorem~\ref{th:quadraticfunction} are valid by replacing $z$, $z_0$, $z_T$ by $\tilde{z}$, $\tilde{z}_0$, $\tilde{z}_T$, respectively. The cost function becomes: 
$S(t_0,z_0,T,z_T) = \frac{1}{\alpha}[z\cdot p]_{t_0}^T+J(z_T)-u_1(T-t_0)$. 
\end{corollary}

\begin{corollary} \label{cor:barycentre}
If the user traffic intensity is of the form $u(t,z)=\sum_i u_i||z-z_i||^2$ with $\sum_i u_i\neq 0$, then $u(t,z)=(\sum_i u_i)||z-z_b||^2+\sum_i u_i||z_i-z_b||^2$ with $z_b=\frac{\sum_i u_iz_i}{\sum_i u_i}$. Define $\tilde{z}=z+z_b$, $\tilde{z}_0=z_0+z_b$, $\tilde{z}_T=z_T+z_b$, $\tilde{u}_0=\sum_i u_i$ and trajectories given in Theorem~\ref{th:quadraticfunction} are valid by replacing $z$, $z_0$, $z_T$, $u_0$ by $\tilde{z}$, $\tilde{z}_0$, $\tilde{z}_T$, $\tilde{u}_0$ respectively. 
\end{corollary}
The system is thus equivalent to the one assumed in Theorem~\ref{th:quadraticfunction} by changing the origin of the locations to the barycentre $z_b$ of the $z_i$.

\subsubsection{Traffic Hot Spot, Traffic Hole}
We assume that there is a hot spot or a traffic hole located in $z_h$ and that the traffic intensity is of the form $u(t,z)=\frac{1}{2}u_0||z-z_h||^2+u_1=\frac{1}{2}u_0||z||^2-u_0z\cdot z_h+\frac{1}{2}u_0||z_h||^2 +u_1$. We can apply Corollary~\ref{cor:quad} with $e=-z_h$.
Figure~\ref{fig:hotspotu0negative} shows optimal trajectories when $z_h$ is a hot spot, i.e., for $u_0<0$, and different values of $K$. The starting point is $z_0$ and the destination is $z_T$. When $K$ increases, the velocity cost increases and the trajectories tend to the straight line between $z_0$ and $z_T$, which minimizes the speed. When $K$ is small, the UAV can go fast to $z_h$, reduces its speed in the vicinity of the hot spot and then goes fast to the destination. %The hot spot plays the role of a repulsor. 
\begin{figure}[t]
\begin{center}
\includegraphics[width=0.8\linewidth]{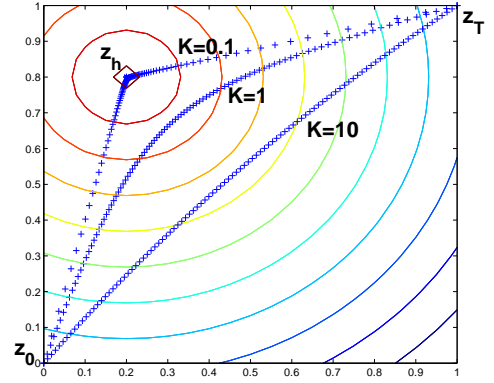}
\end{center}
\caption{\label{fig:hotspotu0negative}Traffic hot spot ($u_0<0$). Circles are iso-traffic levels.}
\end{figure} 
Figure~\ref{fig:u0positive} shows optimal trajectories when $z_h$ is a traffic hole, i.e., for $u_0>0$. In Figure~\ref{fig:u0positive2}, $T$ is smaller than the period of the ellipse, i.e., $\frac{2\pi}{\omega}>T$. 
%Contrary to the case where $u_0<0$, $z_h$ plays the role of a attractor. 
When $K$ decreases, the UAV can spend more time in the areas of higher traffic intensity. In Figure~\ref{fig:u0positive-ellipse}, $T$ is larger than the period. In this case, the trajectory follows one period of the ellipse whose equation is given by (\ref{eq:ellipse}) plus a part of the same ellipse from $z_0$ to $z_T$. 

\begin{figure}[t]
\centering
\subcaptionbox{\label{fig:u0positive2}$T$ is smaller than the ellipse period.}{
	\includegraphics[width=0.8\linewidth]{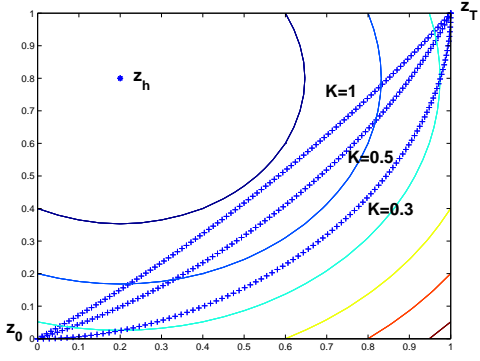}}
\subcaptionbox{\label{fig:u0positive-ellipse}$T$ is larger than the ellipse period.}{
	\includegraphics[width=0.8\linewidth]{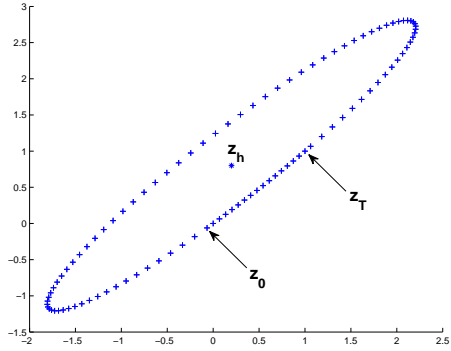}}	
\caption{\label{fig:u0positive}Traffic hole ($u_0>0$).}
\end{figure}

\begin{comment}
\subsubsection{Attractive Serving Base Station} We now assume that the UAV BS is served by a ground BS located in $z_{bs}$ through a backhaul link. The farther the UAV from the base station, the lower the data rate on the backhaul link. We thus add a cost which is proportional to the square distance to the base station. The potential energy now depends on this additional cost: $u(t,z)=\frac{1}{2}u_0||z-z_h||^2+\frac{1}{2}u_2||z-z_{bs}||^2+u_1$ with $u_0+u_2\neq 0$. Applying Corollary~\ref{cor:barycentre}, the system is equivalent to the single hot spot model with hot spot location $z'_h=\frac{u_0z_h+u_2z_{bs}}{u_0+u_2}$. This model is attractive if $u_0+u_2<0$ and repulsive otherwise. An example of such a trajectory is shown in Figure~\ref{fig:barycentre}.

\begin{figure}[t]
\begin{center}
\includegraphics[width=\linewidth]{barycentre}
\end{center}
\caption{\label{fig:barycentre}Attractive serving base station with $u_0+u_2<0$.}
\end{figure} 
\end{comment}

\subsection{Multi-Phase Trajectory Characterization} \label{sec:interface}

We now consider a traffic intensity (or potential) consisting in two quadratic functions separated by an interface $\mathcal{I}$ of equal potentials delimiting two regions $1$ and $2$. The interface is defined by an equation $f(z)=C$, where $C$ is a constant and $f$ is a differentiable function. 
%The optimal trajectory is supposed to cross the interface only once at position $\xi$ at time $\tau$.
We assume that the optimal trajectory crosses only once the interface at position $\xi$ at $\tau$.

\begin{theorem} \label{th:multiphase}
The location and time $(\xi, \tau)$ of interface crossing are characterized by the following equations
\begin{eqnarray}
H_1(\xi(\tau),p^*(\tau^-)) - H_2(\xi(\tau),p^*(\tau^+))&=& 0 \label{eq:Hconserved} \\
p^*(\tau^-)-p^*(\tau^+)-\mu\ \nabla_z f(\xi)&=&0 %\notag 
\label{impulsion-ms:eq}
\\
f(\xi)&=&C \label{eq:interface}
\end{eqnarray}
for some Lagrange multiplier $\mu\in\mathbb{R}$, where we recall that $p^*$ is defined with respect to the optimal trajectory between $(t_0,z_0)$ and $(T,z_T)$, and where $p^*(\tau^-)=\lim_{s\to \tau, s<\tau}p^*(s)$ and $p^*(\tau^+)=\lim_{s\to \tau, s>\tau}p^*(s)$.
\end{theorem} 

\begin{proof}
See Appendix~\ref{app:multiphase}.
\end{proof}

Equation \eqref{eq:Hconserved} expresses the fact the energy is conserved when crossing the interface. One can show that actually the energy is conserved along the whole trajectory. Equation \eqref{impulsion-ms:eq} is related to the conservation of the tangential component of the impulsion at the interface. Equation \eqref{eq:interface} is the interface equation at $\xi$. 
One can show that under the assumption of equal potential on the interface, the kinetic energy, the impulsion, and the velocity vector are conserved across the interface.

% \begin{proposition}
% If the Lagrangian has the form $\mathcal{L}_i(z,a)=T(a)-u_i(z)$, $i=1,2$, i.e., the kinetic energy is the same in both regions, then the total impulsion is conserved when crossing the interface.  
% \end{proposition}

\section{Algorithms} \label{sec:algos}

\subsection{An Online Algorithm: MPC}
In this section, we first present an online algorithm based on MPC~\cite{camacho2013model} (we omit the pseudo-code for space reasons). In a traffic intensity landscape made of multiple phases, the idea is to assume at every $t$ that the current phase won't change from $t$ to $T$. Using this assumption, we compute the optimal trajectory as in the single phase case and take the next decision based on this trajectory. This algorithm has the advantage of being online, of low complexity and can be used in multiphase time-dependent traffic maps. We have however no guarantee of optimality.

\subsection{An Alternating Optimization Algorithm}

We now study a time-independent bi-phase scenario, in which a trajectory from $z_0$ to $z_T$ crosses the interface at time $\tau$ and location $\xi$. We present an Alternating Optimization Algorithm (Algorithm~\ref{alg:aoa}) that provides a stationary trajectory in the sense of Theorem~\ref{th:multiphase}. The algorithm consists in alternatively optimizing $\tau$ (steps 9-17) and $\xi$ (steps 18-26) by using the results of Theorem~\ref{th:multiphase}. For every fixed $\tau$ and $\xi$, the current trajectory is the concatenation of the optimal trajectory between $(t_0,z_0)$ and $(\tau,\xi)$ and the optimal trajectory between $(\tau,\xi)$ and $(T,z_T)$ (step 27). Every iteration of the algorithm only requires the evaluation of two Hamiltonians or the computation of a point $B$, see \eqref{location-B-ms:eq}, and its projection on the interface. Therefore the complexity of an iteration is  $O(1)$. In simulations, MPC is used to produce an initial trajectory.

\def\HSPONE{z_{h1}} % HOTSPOT 1
\def\HSPTWO{z_{h2}} % HOTSPOT 2
\subsubsection{Procedure for seeking an optimal $\tau$ given a fixed  $\xi$}
%\\
We use the result of Theorem~\ref{th:multiphase}. As shown in its proof \cite{UAVbs}, the gradient of $S$ with respect to $\tau$ is given by $H_2(\xi,p^*(\tau^+))-H_1(\xi,p^*(\tau^-))$. We can thus compute the Hamiltonians in every region by differentiating the cost function \eqref{eq:vu0neg} with respect to the final time in region 1 (see \eqref{Hamilton-Jacobi-forward-ms:eq}) and with respect to the initial time in region 2 (see \eqref{Hamilton-Jacobi-backward-ms:eq}). We then update $\tau$ by using a simple gradient descent scheme in step 11. 

\begin{figure}[t]
\begin{center}
\includegraphics[width=0.8\linewidth]{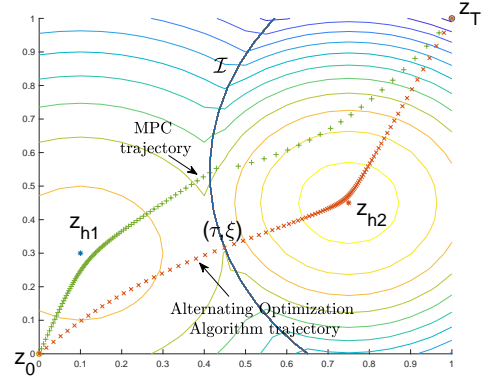}
\end{center}
\caption{\label{fig:mpc-vs-aoa}MPC trajectory and Alternating Optimization Algorithm trajectory with two hot spots.}
\end{figure} 

\begin{figure}[H]
\begin{center}
\includegraphics[width=0.8\linewidth]{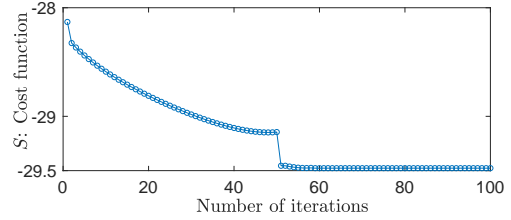}
\end{center}
\caption{\label{fig:costfunction-tau-scenario4}Cost function along the iterations of the Alternating Optimization Algorithm trajectory.}
\end{figure}

\subsubsection{Procedure for seeking an optimal $\xi$ given a fixed  $\tau$}
From Hamilton-Jacobi, the gradient of the total cost function with respect to $\xi$ is
$%Q(\xi, \tau) = 
%\nabla_{\xi} S()  =
p^*(\tau^-) - p^*(\tau+)$
(see proof of Theorem~\ref{th:multiphase} in \cite{UAVbs}).
Since in the Newtonian framework the impulsion is proportional to the control variable $a$ (see \eqref{eq:impulsionquad}) and since in a quadratic model the velocity vector is, at any time a linear combination of {\em centered} initial and final positions (\ref{eq:speedhotspot}), this gradient appears to be an {\em affine} function of $\xi$ which reads
\[
%\nabla_{\xi}% S() 
%(S_1(t_0,z_0,\tau,\xi)+S_2(\tau,\xi,T,z_T))
\nabla_{z_T}S_1(t_0,z_0,\tau,\xi)+\nabla_{z_0}S_2(\tau,\xi,T,z_T) \
= K h\ (\xi - B)
\]

Scalar Hessian $h$ and position $B$, where the spatial gradient cancels {\em i.e.}, $p^*(\tau^-) = p^*(\tau+)$ at fixed $\tau$
are given by: %By using the velocity vector expression and identifying the position $B$ and the scalar Hessian $h$, we can easily 
%compute %are given by
%write explicitly as 
%in closed form
%have the following  expression:
%have closed form 
%%
\begin{equation}
h = \omega_1\, \coth(\omega_1(\tau -t_0)) + \omega_2\, \coth(\omega_2(T - \tau))
\end{equation}
\begin{eqnarray}
B\!\!\!\!&=&\!\!\!\!
\dfrac{1}{h}\big[\
\omega_1\, \HSPONE\, \coth(\omega_1 (\tau - t_0)) + \omega_2\, \HSPTWO\, \coth(\omega_2(T - \tau)) \notag \\
&&+\dfrac{\omega_1\, (z_0 - \HSPONE )}{\sinh(\omega_1 (\tau - t_0))} + \dfrac{\omega_2\,  (z_T - \HSPTWO)}
      {\sinh (\omega_2 (T - \tau))} % corrected 29-06018
      \label{location-B-ms:eq}
\ \big]
\end{eqnarray}
The equation involving the Lagrange multiplier %in
(\ref{impulsion-ms:eq}) new reads
%noted $\mu$:
%
\begin{eqnarray}
%\min_{\xi} S(...)\ s.t.\  \xi \in {\cal I}
%&\Leftrightarrow &
%\nabla_{\xi} S(...) - \mu\ \nabla_{\xi} %f(\xi) = 0
%\\
%Leftrightarrow &
K\ h\ (\xi - B) - \mu\ \nabla_{\xi} f(\xi) = 0
\label{algo-proj-B:eq}
\end{eqnarray}
and shows that the optimal location $\xi^*$ 
is the {\em orthogonal projection} of %the  point 
$B$
on the interface 
${\cal I}$.
%\\
This projection is performed in steps 19-20 of the algorithm.

Figure~\ref{fig:mpc-vs-aoa} shows the MPC trajectory and the trajectory obtained from Algorithm~\ref{alg:aoa} after 60 iterations in a bi-phase landscape. The traffic intensity is shown in three dimensions in Figure~\ref{fig:systemmodel}: It is a bi-phase landscape made of two hot-spots, where the peak of traffic in $z_{h2}$ is higher than in $z_{h1}$. The Alternating Optimization Algorithm has gradually moved the interface crossing time and location in order to spend more time in the second hot-spot and to go closer to $z_{h2}$. 
Figure~\ref{fig:costfunction-tau-scenario4} shows how the cost function has decreased along the iterations and thus how our algorithm has improved over the MPC solution. From iterations 1 to 45, $\tau$ has been gradually updated; at iteration 46, $\xi$ is updated once; $\xi$ is again updated once at iteration 59. 
\begin{comment}
In Figure~\ref{fig:potentialenergy-speed-scenario4} (top), we see how the difference between the potential energy before and after the interface tends to zero. Figure~\ref{fig:potentialenergy-speed-scenario4} (bottom) shows the difference between the velocity vectors before and after crossing the interface. On the resulting stationary trajectory, velocity is continuous on the interface.  
\end{comment}

\algdef{SE}[DOWHILE]{Do}{doWhile}{\algorithmicdo}[1]{\algorithmicwhile\ #1}%
\begin{algorithm}
\caption{Alternating Optimization Algorithm}\label{alg:aoa}
\begin{algorithmic}[1]
\State {\bf Input:} $t_0$, $T$, $z_0$, $z_T$, $z_{h1}$, $z_{h2}$, $u_{01}$, $u_{02}$, $\omega_1$, $\omega_2$, $u_{11}$,  $u_{12}$, an initial trajectory $z(t)$, the initial crossing time and position $(\tau, \xi)\in [\tau;T]\times \mathcal{I}$, $\delta \tau>0$, $\epsilon_{\tau}>0$, $\epsilon_{\xi}>0$, $\epsilon_{S}>0$.
\State {\bf Output:} $(\tau, \xi)\in [\tau;T]\times \mathcal{I}$  such that the conditions of Theorem~\ref{th:multiphase}
\State $\tau'\gets \tau$; $\xi'\gets \xi$
\State $\texttt{timenotfound} \gets 1$; $\texttt{positionnotfound} \gets 0$
\State $\{z(t)\}_{t_0\leq t\leq T}\gets $ an initial feasible trajectory, e.g. from MPC
\State Compute $S$ along $\{z(t)\}_{t_0\leq t\leq T}$
\Do
	\State $S'\gets S$
	\If{$\texttt{timenotfound}$}
		\State Compute $H_1$ and $H_2$ at $(\tau, \xi)$ according to (\ref{Hamilton-Jacobi-backward-ms:eq}-\ref{Hamilton-Jacobi-forward-ms:eq})
		\State $\tau \gets \tau + sign(H_1-H_2)\delta \tau$
		\If{$|\tau'-\tau|<\epsilon_{\tau}$}
			\State $\texttt{timenotfound}\gets 0$
			\State $\texttt{positionnotfound} \gets 1$
		\EndIf
		\State $\tau'\gets \tau$
	\EndIf
	\If{$\texttt{positionnotfound}$}
		\State Compute $B$ according to (\ref{location-B-ms:eq})
		\State $\xi \gets proj_{\mathcal{I}}(B)$, see (\ref{algo-proj-B:eq}) %$\nabla_{\xi} S=h(\xi-B)$
		\If{$||\xi'-\xi||<\epsilon_{\xi}$}
			\State $\texttt{timenotfound}\gets 1$
			\State $\texttt{positionnotfound} \gets 0$
		\EndIf
		\State $\xi'\gets \xi$
	\EndIf
    \State $\{z(t)\}_{t_0\leq t \leq T}\gets \texttt{OPTTRAJ}(z_{h1}, u_{01}, u_{11}, \omega_1, z_0, t_0,$ $\xi, \tau) \cup \texttt{OPTTRAJ}(z_{h2}, u_{02}, u_{12}, \omega_2, \xi, \tau, z_T, T)$  (\texttt{OPTTRAJ} provides optimal trajectory using \eqref{eq:trajhotspot},\eqref{eq:ellipse})
    %see \texttt{OPTTRAJ} function in Algorithm~\ref{alg:mpc})
    \State Compute $S$ for $\{z(t)\}_{t_0\leq t\leq T}$ according to (\ref{eq:vu0neg}) 
\doWhile{$|S'-S|>\epsilon_S$}
\end{algorithmic}
\end{algorithm}

\begin{comment}
\begin{figure}[t]
\begin{center}
\includegraphics[width=1\linewidth]{potentialenergy-speed-scenario4v2}
\end{center}
\caption{\label{fig:potentialenergy-speed-scenario4}Alternating Optimization Algorithm trajectory: (top) $H_1-H_2$: Difference in potential energy before and after crossing the interface; (bottom) $||a_1-a_2||$: Norm of the difference of velocity vectors before and after crossing the interface.}
\end{figure}
\end{comment}

\section{Conclusion} \label{sec:conclusion}
In this paper, we have proposed a Lagrangian approach to solve the UAV base station optimal trajectory problem. When the traffic intensity exhibits a single phase, closed-form expressions for the trajectory and speed are given. %The impact of a serving ground BS is also taken into account. 
When the traffic intensity exhibits multiple phases, we characterize the crossing time and location at the interface. In a first approach, we propose an online algorithm based on MPC for multi-phase and time-dependent traffic intensity, which allows to take into account the impact of each phase. We then propose an offline Alternating Optimization Algorithm for bi-phase time-independent traffic intensities that provides a stationary trajectory with respect to the crossing time and location on the interface and fulfills the necessary conditions of optimality. Numerical results show that we improve the trajectory obtained with MPC. 

\appendix

\subsection{Proof of Lemma~\ref{lemma:EulerLagrange}} \label{app:EulerLagrange}

%The initial location $z_0$ is fixed, let consider the optimal trajectory and let fix the optimal final position $z(T)$. 
Around the optimal trajectory, the first order variation of $S$ is null.
\begin{eqnarray}
\delta S &=&\int_{t_0}^{T}\delta \mathcal{L}(t,z,a)dt \\ \notag
&=& \int_{t_0}^{T} \left[\nabla_z\mathcal{L}(t,z,a)\cdot\delta z(t) + \nabla_a\mathcal{L}(t,z,a)\cdot \delta a(t)  \right] dt \\ \notag
\end{eqnarray}
We now note that $\delta a=\delta \frac{dz}{dt}=\frac{d(\delta z)}{dt}$. Integrating by part the second term in the integral of $\delta S$, we have
\begin{eqnarray}
\lefteqn{\int_{t_0}^T \nabla_a\mathcal{L}(t,z,a) \cdot \frac{d(\delta z)}{dt}dt}  \\ &=& [\delta z(t) \cdot \nabla_a\mathcal{L}(t,z,a)]_{t_0}^T - \int_{t_0}^T \delta z(t) \cdot\frac{d}{dt}\nabla_a \mathcal{L}(t,z,a)dt \notag
\end{eqnarray}
Note that $[\delta z  \frac{\partial \mathcal{L}}{\partial a}]_{t_0}^T=0$ because $z_0$ and $z_T$ are fixed. 
Equating $\delta S$ to zero gives now
\begin{eqnarray}
0=\int_{t_0}^T \left[\nabla_z\mathcal{L}(t,z,a)-\frac{d}{dt}\nabla_a\mathcal{L}(t,z,a)  \right]\cdot \delta z(t) dt.
\end{eqnarray}
As this should be true for every $\delta z$, $\mathcal{L}$, $z_0$ and $z_T$, we obtain the first result. 

Assume that we have the optimal $a(t)$, the condition for $z(T)$ to be the optimal final position is
\begin{eqnarray}
\delta S & =& [\delta z(t)\cdot  \nabla_a \mathcal{L}(t,z,a)]_{t_0}^T + \nabla J(z(T))\cdot\delta z(T) \notag \\
&=& \nabla_a\mathcal{L}(z(T),T,a(T))\cdot \delta z(T) + \nabla J(z(T))\cdot \delta z(T) \notag \\
&=&0
\end{eqnarray}
Note that $z_0$ is fixed and so $\delta z$ in $z_0$ is null. We thus obtain the second result of the lemma. 

\subsection{Proof of Lemma~\ref{lemma:lagrangienhomogene}} \label{app:lagrangienhomogene}

As $\mathcal{L}(z,a)$ is an homogeneous function of $z$ and $a$, we have: $\mathcal{L}(\lambda z,\lambda a)=|\lambda|^{\alpha}\mathcal{L}(z,a)$ for all $\lambda$ (in our case with $\alpha=2$). Deriving this expression with respect to $\lambda$, setting $\lambda=1$, and noting that $a=\dot{z}$ we obtain
\begin{eqnarray}
z\cdot \frac{\partial \mathcal{L}(z,\dot{z})}{\partial z} +\dot{z}\cdot \frac{\partial \mathcal{L}(z,\dot{z})}{\partial z}&=&\alpha \mathcal{L}(z,\dot{z}). \label{eq:eulerhom}
\end{eqnarray}
Using (\ref{eq:eulerimpulsion}) and (\ref{eq:eulerhom}), we have:  $z\cdot \frac{d p}{dt}+\dot{z}\cdot p=\alpha \mathcal{L}$ or equivalently $\frac{d(p\cdot z)}{dt}=\alpha \mathcal{L}$. We can now integrate the cost function (\ref{eq:problem}) along the optimal trajectory as follows
\begin{eqnarray}
S(t_0, z_0,T,z_T)&=&\frac{1}{\alpha} \int_{t_0}^T \frac{d(p\cdot z)}{dt}(t) dt+J(z_T) \notag \\
%+\mathcal{L}_T(z_T)\\
&=&\frac{1}{\alpha} \left(p(T)\cdot z_T-p(t_0)\cdot z_0\right ) +J(z_T) \notag \nopagebreak
%+\mathcal{L}_T(z_T) \notag
\end{eqnarray}

\subsection{Proof of Lemma~\ref{lemma:hamiltonjacobi}}
\label{app:hamiltonjacobi}

We assume that an optimal trajectory exists and we apply the principle of optimality on the optimal trajectory between $(t,z^*(t))$ and $(t+h,z^*(t)+ah)$, where $h>0$. For simplicity, we omit variables $T$ and $z_T$.
\begin{eqnarray}
S(t,z^*(t)) &=& \min_{a} [h \mathcal{L}(z^*(t),a)+S(t+h,z^*(t)+ah)] \notag \\
&=& \min_{a}[h \mathcal{L}(z,a)+ \notag \\
&& S(t,z^*(t))+ha\cdot \nabla_z S(t,z^*(t))+ \notag \\
&& h\frac{\partial S}{\partial t_0}(t,z^*(t))] \notag \\
\frac{\partial S}{\partial t_0}(t,z^*(t))&=& -\min_{a} [a\cdot \nabla_z S(t,z^*(t))+\mathcal{L}(z^*(t),a)] \notag   \\
&=& \max_a [-a\cdot \nabla_z S(t,z^*(t))-\mathcal{L}(z^*(t),a)] \notag \\
&=& H(t,z^*(t),-\nabla_z S(t,z^*(t))) \notag 
\end{eqnarray}
By using the same approach between $t-h$ and $t$, we deduce in the same way equation \eqref{Hamilton-Jacobi-forward-ms:eq} when the final time $T$ is varying.  
%The $a^*$ that minimizes the right hand side of the first equation verifies: $\nabla_z S(z,t)=\left. \frac{\partial \mathcal{L}(z,a)}{\partial a}\right |_{a=a^*}=p$ by the definition of $p$. As a consequence: $\frac{\partial S}{\partial t}=-[a^*\cdot p - \mathcal{L}(z,a^*)]=-H(t,z,p)$.   

\subsection{Proof of Theorem~\ref{th:quadraticfunction}} \label{app:quadraticfunction}

From (\ref{eq:eulerlagrange2}) and (\ref{eq:Lsinglephase}), we obtain the following ordinary differential equation of second degree: $\ddot{z}=-\frac{u_0}{K}z$. If $\frac{u_0}{K}>0$, we define $\omega^2=\frac{u_0}{K}$ and we look for an optimal trajectory of the form: $z(t)=A\cos(\omega t)+B\sin(\omega t)$. If $\frac{u_0}{K}<0$, we look for an optimal trajectory of the form: $z(t)=A\cosh(\omega t)+B\sinh(\omega t)$ with $\omega^2=-\frac{u_0}{K}$. Let us denote $z_0=z(t_0)$ and $a_0=a(t_0)$ the initial conditions for $z$ and $\dot{z}$. 

Take the case $\frac{u_0}{K}<0$. Using the derivative of $z(t)$ and identifying terms, we obtain: $z(t)=z_0\cosh \omega(t-t_0) + \frac{a_0}{\omega}\sinh \omega(t-t_0)$. At $t=T$, we have also: $z_T=z_0\cosh \omega(T-t_0)+\frac{a_0}{\omega}\sinh \omega(T-t_0)$, from which we deduce
\begin{eqnarray}
a(t_0)&=&\frac{\omega(z_T-z_0\cosh \omega(T-t_0))}{\sinh \omega(T-t_0))} \\
a(T)&=&\frac{\omega(-z_0+z_T\cosh \omega(T-t_0))}{\sinh \omega(T-t_0))}
\end{eqnarray}
when $u_0<0$. In a similar way, we have:
\begin{eqnarray}
a(t_0)&=&\frac{\omega(z_T-z_0\cos \omega(T-t_0))}{\sin \omega(T-t_0)} \\
a(T)&=&\frac{\omega(-z_0+z_T\cos \omega(T-t_0))}{\sin \omega(T-t_0)}
\end{eqnarray}
when $u_0>0$. Injecting $a(t_0)=a_0$ in the equation of the trajectory provides the result.

For the computation of $S$, we now use the result of Lemma~\ref{lemma:lagrangienhomogene} as our cost function is 2-homogeneous.
From equation (\ref{eq:valuefunction}), we see that only initial and final conditions are required to compute the cost function. Recall now that $p=Ka$. 
Injecting the equations of $a(t_0)$ and $a(T)$ in (\ref{eq:valuefunction}), we obtain the result for the cost function. 
%
%Let note first that we have the ordinary differential equation of second degree $\ddot{z}=-\frac{u_0}{K}z$ from (\ref{eq:eulerlagrange}). If $\frac{u_0}{K}<0$, we look for an optimal trajectory of the form: $z(t)=A\cosh(\omega t)+B\sinh(\omega t)$ with $\omega^2=-\frac{u_0}{K}$. Let us denote $z_0$ and $a_0$ the initial conditions for $z$ and $\dot{z}$. We thus have with $\frac{u_0}{K}<0$: $z(t)=z_0\cosh \omega(t-t_0) + \frac{a_0}{\omega}\sinh \omega(t-t_0)$. At $t=T$, we have: $z_T=z_0\cosh \omega(T-t_0)+\frac{a_0}{\omega}\sinh \omega(T-t_0)$, from which we deduce:
%\begin{equation} 
%a_0=\frac{\omega}{\sinh \omega (T-t_0)}(z_T-z_0\cosh \omega(T-t_0))
%\end{equation}
%Using the same approach with the final conditions $z_T$ and $a_T$, we obtain:
%\begin{equation}
%a_T=\frac{\omega}{\sinh \omega (T-t_0)}(z_T\cosh(T-t_0)-z_0)
%\end{equation}
%As a result:
%$v(z_0,t_0)=\frac{K}{\alpha}(\left(a_Tz_T-a(t_0)z_0\right )+c_T(z_T)$ provides the solution.
%
%If $\frac{u_0}{K}>0$, we define $\omega^2=\frac{u_0}{K}$ and we look for an optimal trajectory of the form: $z(t)=A\cos(\omega t)+B\sin(\omega t)$. The rest of proof is similar. 

\subsection{Proof of Theorem~\ref{th:multiphase}}
\label{app:multiphase}

We assume that the location and time $(\xi,\tau)$ of interface crossing is known and unique. The optimal trajectory between $(z_0,t_0)$ and $(z_T,T)$ can be decomposed in two sub-trajectories that are themselves optimal between $(z_0,t_0)$ and $(\xi,\tau)$ on the one hand and between $(\xi,\tau)$ and $(z_T,T)$ on the other hand, by the principle of optimality. 
In region 1, the optimal cost up to $\tau$ is
\begin{equation}
S_1(t_0,z_0,\tau,\xi)=\int_{t_0}^{\tau}\mathcal{L}(z^*(s),a^*(s))ds
\end{equation}
Using Hamilton-Jacobi, we obtain
\begin{equation}
\frac{\partial S_1}{\partial T}(t_0,z_0,\tau,\xi)=-H_1(\xi,p^*(\tau^-)). 
\end{equation}
In the same way, the optimal cost in region 2 is
\begin{eqnarray}
S_2(\tau,\xi,T,z_T)&=&\int_{\tau}^{T}\mathcal{L}(z^*(s),a^*(s))ds 
\end{eqnarray}
Using again Hamilton-Jacobi, we obtain
\begin{equation}
\frac{\partial S_2}{\partial t_0}(\tau,\xi,T,z_T)=H_2(\xi,p^*(\tau^+))
\end{equation}
The total cost along the optimal trajectory is the sum of the cost over the two regions
\begin{equation} \label{eq:sumcost}
S(t_0,z_0,T,z_T)=S_1(t_0,z_0,\tau,\xi)+S_2(\tau,\xi,T,z_T)
\end{equation}
% % S^{*(2)}(T;\xi,\tau,z_T,T)&=&\int_{\tau}^{T}\mathcal{L}(z^*,t',a^*)dt' \notag \\
% % &=& -S^{*(2)}(\tau;z_T,T,\xi,\tau)
% % \end{eqnarray}
% where the second equation is obtained along the same trajectory taken in reverse way. \textcolor{red}{This hold provided that the reverse trajectory is also optimal, which is the case when the potential $u(z,t)$ does not depend explicitly on $t$.}
% We now use Hamilton-Jacobi on the optimal trajectory in zone 2:
% \begin{eqnarray}
% \left. \frac{\partial S^{*(2)}(t;\xi,\tau,z_T,T)}{\partial t}\right |_{t=T}&=& -\left. \frac{\partial S^{*(2)}(t;z_T,T,\xi,\tau)}{\partial t}\right |_{t=\tau} \notag \\
% &=&-H^{(2)}(p,\xi,\tau;z_T,T,\xi,\tau) \notag \\
% &=&-H^{(2)}(p,\xi,\tau;\xi,\tau,z_T,T) \notag
% \end{eqnarray}
% where $p=p^{(2)}(\xi,\tau)$ is the impulsion at $\tau$ in zone 2. \textcolor{red}{The last equation comes from the fact that the Lagrangian does not change when we take the reverse trajectory and the Legendre transform of a pair function is pair.}
% Now the optimal cost along the whole trajectory can be written:
% \begin{eqnarray}
% S^{*}_{\tau}(z_0,t_0,z_T,T)&=&S^{*(1)}(\tau;z_0,t_0,\xi,\tau)+S^{*(2)}(T;\xi,\tau,z_T,T) \notag \\
% &=&S^{*(1)}(\tau;z_0,t_0,\xi,\tau)-S^{*(2)}(\tau;z_T,T,\xi,\tau) \notag
% \end{eqnarray}
A necessary condition for the optimality of $\tau$ is thus
\begin{equation}
\frac{\partial S_1}{\partial T}(t_0,z_0,\tau,\xi)+\frac{\partial S_2}{\partial t_0}(\tau,\xi,T,z_T)=0,
\label{optimal-tau-first-ms:eq}
\end{equation}
i.e.,
\begin{equation}
H_1(\xi,p^*(\tau^-))=H_2(\xi,p^*(\tau^+))\label{optimal-tau-second-ms:eq}
\end{equation}
A necessary condition for the optimality of $\xi$ in \eqref{eq:sumcost} under the constraint $f(\xi)=C$ is characterized by
\begin{eqnarray}
\mu \nabla_z f(\xi)&=&\nabla_{z_T}S_1(t_0,z_0,\tau,\xi)+\nabla_{z_0}S_2(\tau,\xi,T,z_T) \notag
\\ &=&
p^*(\tau^-) - p^*(\tau^+) \notag
\end{eqnarray}
where $\mu$ is a Lagrange multiplier associated to the constraint and where
the second line comes from
equation (\ref{eq:optimpulsion})
of Hamilton-Jacobi. 
% From \eqref{eq:optimpulsion}, we obtain the second equation.
Thus we obtain precisely equation (\ref{impulsion-ms:eq}).

\bibliographystyle{IEEEtran}
\bibliography{IEEEabrv,dronebs}

\end{document}